\documentclass[12pt]{article}

\usepackage{indentfirst}
\usepackage{amssymb}
\usepackage{verbatim}
\usepackage{amsthm}
\usepackage{fleqn}
\usepackage{graphicx}
\usepackage{epstopdf}
\usepackage{amsfonts}
\usepackage[numbers]{natbib}
\usepackage[colorlinks,citecolor=blue,urlcolor=blue]{hyperref}
\usepackage{amsmath}

\usepackage{geometry}
\usepackage{pifont}
\newtheorem{theorem}{Theorem}[section]
\newtheorem{lemma}[theorem]{Lemma}

\newtheorem{remark}[theorem]{Remark}

\def\EE{\mathbb{E}}\def\FF{\mathbb{F}}\def\HH{\mathbb{H}}
\def\LL{\mathbb L}
\def\RR{\mathbb{R}}\def\SS{\bb S}
\def\cL{{\cal L}}
\def\cD{{\cal D}}\def\cF{{\mathcal F}}\def\cG{{\cal G}}

\def\cL{{\cal L}}
\def\cP{{\cal P}}\def\cQ{{\cal Q}}\def\cR{{\cal R}}\def\cU{{\cal U}}

\def\al{{\alpha}}\def\be{{\beta}}\def\de{{\delta}}
\def\ep{{\epsilon}}\def\ga{{\gamma}}
\def\la{{\lambda}}\def\si{{\sigma}}

\def\Om{{\Omega}}

\def\<{\left<}\def\>{\right>}\def\({\left(}\def\){\right)}

\def\Om{{\Omega}}
\def\<{\left<}\def\>{\right>}\def\({\left(}\def\){\right)}

\newfam\msbmfam\font\tenmsbm=msbm10\textfont
\msbmfam=\tenmsbm\font\sevenmsbm=msbm7
\scriptfont\msbmfam=\sevenmsbm\def\bb#1{{\fam\msbmfam #1}}

\def\HH{\mathbb H}

\def\al{{\alpha}}\def\be{{\beta}}\def\de{{\delta}}
\def\ep{{\epsilon}}\def\ga{{\gamma}}
\def\la{{\lambda}}\def\si{{\sigma}}

\def\Om{{\Omega}}

\def\<{\left<}\def\>{\right>}\def\({\left(}\def\){\right)}

\def\cF{{\cal F}}\def\cG{{\cal G}}
\def\cL{{\cal L}}\def\cP{{\cal P}}

\begin{document}

\title{Mean-field stochastic linear quadratic control problem with random coefficients\thanks{This work was supported by
the National Key R\&D Program of China (2022YFA1006102), and
the National Natural Science Foundation of China (12471418, 12326368).}}

\author{Jie Xiong\thanks{Department of Mathematics and Shenzhen International Center for Mathematics,
		Southern University of Science and Technology, Shenzhen 518055, China
		(xiongj@sustech.edu.cn).}
\and Wen Xu\thanks{Department of Mathematics,
	Southern University of Science and Technology, Shenzhen 518055, China
	(12231279@mail.sustech.edu.cn).
}}


\maketitle

\begin{abstract}
In this paper, we first prove that the mean-field stochastic linear quadratic (MFSLQ for short) control problem with random coefficients has a unique optimal control
and derive a preliminary stochastic maximum principle to characterize this optimal control by an optimality system.
 However, because of the term of the form 
$\mathbb{E}[A_1(\cdot)^\top Y(\cdot)] $ in the adjoint equation, which cannot be represented in the form
$\mathbb{E}[A_1(\cdot)^\top]\mathbb{E} [Y(\cdot)] $,   we cannot solve this optimality system explicitly. To this end, we decompose the MFSLQ control 
problem into two  problems without the mean-field terms, and one of them is a constrained problem. The constrained SLQ control problem
is solved explicitly 
by an extended LaGrange multiplier method developed in this article.

\end{abstract}

\text{\bf Keywords}: Extended LaGrange multiplier method, mean-field control, linear quadratic control problem, random coefficient, Riccati equation.

\text{\bf AMS Subject Classification}: 49N10, 60H10, 93E20.

\baselineskip=18pt

\section{Introduction}
\setcounter{equation}{0}
\renewcommand{\theequation}{\thesection.\arabic{equation}}

Let $(\Omega,\mathcal{F},\mathbb{F},\mathbb{P})$ be a complete filtered probability space on which  an
 one-dimensional standard Brownian motion $\{W(t); 0 \leq t < \infty\}$ is defined, where  $\mathbb{F} = \{\mathcal{F}_t\}_{t\geq 0}$ is the natural   filtration generated by $W(\cdot)$  augmented by all the $\mathbb{P}$-null sets.    We consider the following controlled  linear   mean-field    stochastic  differential equation (MFSDE  for short) with random coefficients:
\begin{equation}\label{OState1}
	\begin{cases}
		dX(s)=[A(s)X(s)+A_1(s)\mathbb{E}X(s)+B(s)u(s)]ds\\
		\quad \quad \quad \quad +[C(s)X(s)+C_1(s)\mathbb{E}X(s)+D(s)u(s)]dW(s), & s\in [0,T],\\
		X(0)=x, \\
	\end{cases}
\end{equation}
where $A(\cdot), A_1(\cdot), C(\cdot), C_1(\cdot): [0,T]\times \Omega \to \mathbb{R}^{n\times n}$ and $ B(\cdot), D(\cdot) : [0,T]\times \Omega \to  \mathbb{R}^{n\times m}$, are matrix-valued $\cF_t$-adapted processes, and  $u(\cdot)$ is an $\cF_t$-adapted process satisfying 
$\mathbb{E}\int_0^T|u(s)|^2ds \linebreak[4]< \infty$. The initial state $x \in \RR^n$ will be fixed throughout this article.
 In the above, $X(\cdot)$  valued in $\mathbb{R}^n$  is the state process, and $u(\cdot)$  valued in $\mathbb{R}^m$    is the control process. We will denote $X(\cdot)$ by $X^u(\cdot)$ when it is necessary to indicate the control it corresponds.

We define the following cost functional with random coefficients 
\begin{equation}\label{Cost2}
	\begin{aligned}
		J\big( u(\cdot)\big)  =&\mathbb{E}\Big\{\int_{0}^{T}\Big(\langle Q(s)X(s),X(s)\rangle+\langle Q_1(s)\mathbb{E}X(s),\mathbb{E}X(s)\rangle \\
		&+\langle R(s)u(s),u(s)\rangle \Big)ds  + \langle GX(T),X(T)\rangle\Big\}, 
	\end{aligned}
\end{equation}
where  $Q(\cdot), Q_1(\cdot) : [0,T]\times \Omega \to \mathbb{R}^{n\times n} $,  $R(\cdot) :  [0,T]\times \Omega \to \mathbb{R}^{m\times m}$, and $G$ is an $\mathcal{F}_T$-measurable random matrix. For a control $u(\cdot)$ belonging to the following space
$$
\mathcal{U} =  \left\{u(\cdot):[0, T] \times \Omega \rightarrow \mathbb{R}^m \mid u(\cdot)  \text { is $\cF_t$-adapted and }
 \mathbb{E} \int_0^T|u(t)|^2 d t<\infty\right\},
$$
the MFSLQ  optimal control problem with random coefficients  can be stated as follows:

\vspace{2mm}
\textbf{Problem (MFSLQ)}:
{\em Find a control $u^\ast(\cdot) \in \mathcal{U}$ such that
\begin{equation}\label{aimc}
	J\big( u^\ast(\cdot)\big)= \inf_{u(\cdot) \in \mathcal{U}}J\big( u(\cdot)\big).
\end{equation} 
 The process $u^\ast(\cdot)$ is called an optimal control of Problem (MFSLQ). The corresponding  process $X^\ast (\cdot )$ is called an  optimal state, and  $\big(X^\ast(\cdot), u^\ast (\cdot)\big)$ an  optimal pair.}

MFSDEs were initially used to describe physical systems involving a large number of interacting particles. The complexity in the dynamics of an SDE   is reduced by replacing the interactions of all particles with their expectation. 
 The study of the optimal control problem of the mean-field system has gained popularity in the last ten years since Buchdahn $et\ al$ (\cite{Buckdahn2009}, \cite{Buckdah2009}) and Carmona and Delarue (\cite{Carmona2013a}, \cite{Carmona2013b}, \cite{Carmona2015})
  introduced the mean-field  backward stochastic differential equation (BSDE for short) and mean-field forward-backward stochastic differential 
  equation (FBSDE for short). However, there is a shortage of literature on studying  Problem (MFSLQ) with random coefficients. 
    In the following, we will discuss  the degeneration
     of this problem and the development of the related research. These results  provide an affluent theoretical 
     foundation for  our research.

Problem (MFSLQ) degenerates into a classical SLQ control problem when the mean-field terms both in the state equation (\ref{OState1}) and in the cost functional (\ref{Cost2}) vanish and all the coefficients are deterministic. The study of the SLQ problem  for deterministic coefficients   can be  traced back to the work of    
 Kalman \cite{kalman19601} and  the related work of Kushner \cite{Kushner1962}, Davis \cite{Davis1977}, and Wonham \cite{Wonham1968}.   It is well known that  this  theory is rather well-developed under the assumptions that the weighting matrix $R$ is positive definite and the weighting matrices  $G$ and $Q$  are positive semi-definite. In this case, there is only one optimal control, which can be represented  in  the state feedback form by the  solution of the Riccati equation; see the book \cite{Yong1999}   for more details.     

The next step is to consider  Problem (MFSLQ) without mean-field terms but with  random coefficients.  
 To our best knowledge, the  SLQ  problem  with random coefficients at least  dates back to Bismut in the
1970s, see \cite{Bismut1976} and \cite{Bismut1978}, {\color{blue} where he derived the associated Riccati equation as a matrix-valued system of quadratic BSDEs, and left its existence and uniqueness as a big challenging open problem. Subsequently in 1999, the existence and uniqueness of this stochastic Riccati equation, was listed as one open problem on BSDEs in \cite{Peng1998}, and it was completely solved in 2003 by Tang \cite{tang2003} with the help of the associated forward-backward stochastic Hamilton flows (see also Kohlmann and Tang \cite{KT1}, \cite{KT2} for a different method and an overview on this open problem).  
  The open problem itself  and the  presentation in  \cite{tang2003} were for the regular case. }The singular 
case was mentioned at the end of \cite{tang2003} and solved in \cite{tang2002}.
The extension to the indefinite situation was given by Sun $et\ al$ \cite{sun2021}.   We   emphasize that the  work \cite{tang2003}  on the SLQ problem with random coefficients plays an important role in our current study. 

 For the case where all the coefficients are deterministic, Problem (MFSLQ)  is reduced to the case covered by Yong  \cite{Yong2013}.  
 It was
  pointed out there  that the deterministic  coefficients  have  played an essential role in dealing with the problem, and there is a
    lack of  method to deal with the case with random coefficients.  Moreover, many authors have made contributions to the general deterministic setup,  where the state equation is a nonlinear SDE,  and  the cost functional is also non-quadratic.   For example, the work of   Buckdahn $et\ al$ \cite{Buckdahn2011},  Andersson and Djehiche \cite{Andersson 2011}, Meyer-Brandis $et\ al$ \cite{Meyer2012}, to name just a few. 
 
We point out that  the most  relevant work to the current setting   is  Pham \cite{Pham2016}, 
 studying the  SLQ optimal control of $\FF^0$-conditional  MFSDE with random coefficients, where  $\mathbb{F}^0$ is a sub-filtration of  $\mathbb{F}$,
  which is generated  by another Brownian motion $W^0$. However,  these random coefficients are assumed to be $\mathbb{F}^0$-adapted processes.   Notice that the main challenge of   Problem (MFSLQ)  is that   the term of the form $\mathbb{E}[A_1(\cdot)^\top Y(\cdot)] $ appears in the adjoint equation, and it cannot be separated into the form $\mathbb{E}[A_1(\cdot)^\top]\mathbb{E} [Y(\cdot)] $. In the $\FF^0$-conditional mean-field SLQ control problem,  the assumption on  coefficients     essentially avoids such difficulty mentioned 
above since $\EE[A_1(s)^\top Y(s)|\cF^0_s]=A_1(s)^\top\EE[Y(s)|\cF^0_s]$.   Another related work is Mei $et\; al$ \cite{MWY} where $ \mathbb{F}^0$ is 
generated by a Markov chain.

The main contributions of this paper are as follows. 
 First, we establish the stochastic maximum principle (SMP for short) for  Problem (MFSLQ), including the existence and uniqueness of an optimal control.
 More importantly, we propose a new method for studying mean-field control problems. Namely, we decompose
  the MFSLQ problem into two  problems without mean-field terms. The first one is an SLQ control problem with the constraint that the expectation process is a given deterministic function, while the second is an SLQ with deterministic control process. 
  Finally, we solve the constrained problem by an extended LaGrange multiplier (ELM for short) method proposed in this article. 
  We believe that this new approach can be applied to many other mean-field problems.

The rest of the paper is organized as follows. In Section \ref{sec2},  we introduce some spaces and present the main results.
Then, in Section \ref{sec3}, we prove that Problem (MFSLQ) has a unique optimal control, and it satisfies an optimality system.
 Section \ref{sec4} is devoted to  the ELM method to represent the constrained 
 optimal control as a functional of the deterministic function, which is
 the constraint of the expectation process.  We solve the SLQ with deterministic control of the frozen mean-field in Section \ref{sec6}.  Finally, we give some concluding remarks
in Section \ref{sec7}. Throughout this article, we will use $K$ to represent a positive constant whose value can be different from place to place.

\section{Main results}\label{sec2}
\setcounter{equation}{0}
\renewcommand{\theequation}{\thesection.\arabic{equation}}

In this section, we proceed to presenting the main results of this article. First, we introduce some notations and conditions.
Denote the collection of all $n\times n$ positive semi-definite matrices by $\SS^n_+$. For Euclidean space $\mathbb{H} = \mathbb{R}^n,  \mathbb{R}^{m\times n},   \SS^n_+$,   we introduce the following spaces. 
\begin{itemize}
\item $L^{2,c}_\FF(\HH)\equiv L_{\mathbb{F}}^{2}\(\Omega; C([0,T]; \mathbb{H}) \)$: the space of continuous $\cF_t$-adapted processes
	$X: [0, T]\times \Omega \to \mathbb{H}$ with $\mathbb{E}\left[\text{sup}_{0 \leq s \leq T}  \|X(s,\omega)\|_\HH^{2}\right]< \infty$.

\item $L^{p,q}_{\FF}(\HH)\equiv L^p_{\FF}\(\Om;L^q([0,T];\HH)\)$: the space of  $\cF_t$-adapted processes
	$X: [0, T]\times \Omega \to \mathbb{H}$ with $$\mathbb{E}\(\int^{T}_{0}  \|X(s,\omega)\|_\HH^{q}ds\)^p< \infty.$$
Especially, we denote $L^2_{\FF}(\HH)\equiv L^{1,2}_{\FF}(\HH)$.

\item $L^\infty_\FF(\HH)\equiv L_{\mathbb{F}}^{\infty}(0, T;  \mathbb{H})$: the space of  $\cF_t$-adapted $  \mathbb{H}$-valued  bounded processes. 

\item $L^{\infty,c}_\FF(\HH)$: the space of  $\cF_t$-adapted $  \mathbb{H}$-valued  bounded continuous processes.

\item $L_{\cG}^{2}(\mathbb{H})$: the space of $\cG$-measurable $\mathbb{H}$-valued square integrable  random variables, where 
$\cG\subset\cF$ is a sub-$\si$-field in $\Om$.

\item $L_{\cG}^{\infty}(\mathbb{H})$: the space of $\cG$-measurable $\mathbb{H}$-valued bounded  random variables.

\item $\LL^2\equiv L^2([0, T]; \mathbb{R}^n)$: the space of deterministic $\RR^n$-valued square-integrable  functions on $[0, T]$.
\end{itemize}

Throughout this article, we  impose the following assumptions:
\begin{itemize}
	\item[$\mathbf{(H1):}$]
 $ A(\cdot), A_1(\cdot), C(\cdot), C_1(\cdot) \in L_{\FF}^{\infty}(\mathbb{R}^{n\times n}),$ and 
$B(\cdot), D(\cdot) \in  L_{\FF}^{\infty}(\mathbb{R}^{n\times m})$. 
\item [$\mathbf{(H2)}:$] $Q(\cdot), Q_1(\cdot) \in L_{\FF}^{\infty}( \mathbb{S}_{+}^n),$   and $  R(\cdot) \in L_{\FF}^{\infty}\big( \mathbb{S}_{+}^m),$ $G  \in L_{\mathcal{F}_T}^{\infty}(\mathbb{S}_{+}^n).  $  Moreover, there exists a constant $\delta > 0$ such that
$$
R (s)\geq \delta I_{m}, \; a.e.\; s\in [0, T], a.s.,
$$
where $I_m$ is the $m\times m$ identity matrix.

\end{itemize}

The following result is a preliminary SMP.
\begin{theorem}\label{SMP1}
	Let  (H1)  and (H2)  hold. Then, Problem  (MFSLQ) has a unique optimal control.  Further,	
	 $u^\ast(\cdot)$ is the optimal control for  Problem  (MFSLQ)  if and only if the adapted solution 
	 $\big(X^\ast(\cdot), Y^\ast(\cdot), Z^\ast(\cdot)\big)\in\(L^{2,c}_\FF(\RR^n)\)^2\times L^2_\FF(\RR^n)$ to  the  following 	 FBSDE: for $s \in [0, T]$,
	\begin{equation}\label{FBSDEoptimal}
		\begin{cases}
			dX^*(s)=[AX^\ast+A_1\mathbb{E}X^\ast+Bu^\ast]ds+[CX^\ast+C_1\mathbb{E}X^\ast+Du^\ast]dW(s), \\
			dY^\ast(s)=-\{
			A^{\top}Y^\ast+C^{\top}Z^\ast+QX^\ast+\mathbb{E}Q_1\mathbb{E}X^\ast +\mathbb{E}[A_1^{\top}Y^\ast+C_1^{\top}Z^\ast]\}ds+Z^\ast(s)dW(s),\\
			X^\ast(0)=x, \quad Y^\ast(T)=GX^\ast(T)\\
		\end{cases}
	\end{equation}
	admits the stationary condition:
	\begin{equation}\label{eq0921c}
		R(s)u^\ast(s)+ B(s)^{\top}Y^\ast(s)+D(s)^{\top}Z^\ast(s)=0, \; a.e. \; s \in[0, T], a.s.
	\end{equation}
\end{theorem}
Here in the equation (\ref{FBSDEoptimal}) above,  we dropped the dependence of the processes on the time parameter
  $s$ in some complicated terms
 for notation simplicity. We will continue to do so when
no confusion will be raised.

As we mentioned in the introduction, the main difficulty in solving FBSDE (\ref{SMP1}) with condition (\ref{eq0921c}) is that
$\mathbb{E}[A_1(s)^{\top}Y^\ast(s)]\neq \mathbb{E}[A_1(s)^{\top}]\mathbb{E}[Y^\ast(s)]$. We will get around this difficulty by using 
an ELM method which we introduce now.

Denote by $\LL^2_0$ the set of $\al(\cdot)\in\LL^2$ such that 
\[\left\{u(\cdot)\in\cU:\; \EE X^u(\cdot)=\al(\cdot)\right\}\neq\emptyset.\]
Note that
\begin{equation}\label{eq1116c}
\inf_{u(\cdot) \in \mathcal{U}}J\big( u(\cdot)\big)=\inf_{\al(\cdot)\in \LL^2_0}\inf\left\{J(u(\cdot)):\; u(\cdot)\in\cU,\; \EE X^u(\cdot)=\al(\cdot)\right\}.
\end{equation}
First, we consider a problem, called Problem 1, with the constraint that the state process $X(\cdot)$ satisfies $\EE X(s)=\al(s)$, $\forall s\in[0,T]$, for a fixed 
deterministic function $\al(\cdot)\in\LL^2_0$. In this case, we denote the cost functional as $\hat{J}_\al(\cdot)$. Namely,
\begin{equation}\label{eq0503a}
	\begin{aligned}
		\hat J_{\alpha}\(u(\cdot)\) = & \; \mathbb{E} \Big\{ \int_{0}^{T}\Big(\langle Q(s)X(s),X(s)\rangle+\langle Q_1(s)\alpha(s),\alpha(s)\rangle +\langle R(s)u(s),u(s)\rangle \Big) ds  \\
		&+ \langle GX(T),X(T)\rangle 	\Big\}.
	\end{aligned}
\end{equation}

\begin{lemma}\label{lem1114b}
For $\al(\cdot)\in\LL^2_0$  fixed, there exists a unique $u^*_\al(\cdot)\in\cU$ such that $\EE X^{u^*_\al}(\cdot)=\al(\cdot)$ and
\[\hat{J}_\al(u^*_\al(\cdot))=\inf\left\{\hat{J}_\al(u(\cdot)):\; u(\cdot)\in\cU,\; \EE X^u(\cdot)=\al(\cdot)\right\}.\]
\end{lemma}

We relax the constraint by introducing an ELM $\la(\cdot)$, which is a deterministic function.
More specifically, fixing $\al(\cdot)\in\LL^2_0$, we consider the control problem with state equation:
\begin{equation}\label{OState}
	\begin{cases}
		dX(s)=\; [AX+Bu+A_1\alpha]ds +[CX+Du+C_1\alpha]dW(s), & s\in [0,T],\\
		X(0)\;\;= \; x, \\
	\end{cases}
\end{equation}
and the objective functional 
\[J_{\alpha}\big(u(\cdot),\la(\cdot)\big) =\hat{J}_\al(u(\cdot))+2\<\la,\EE X^u-\al\>_{\LL^2},\quad
\forall\  \lambda(\cdot)\in \LL^2,\; u(\cdot)\in\cU.\]

\vspace{2mm}
 \textbf{Problem 1}: Find a control $u_{\alpha}(\cdot) \in \mathcal{U}$ and an ELM $\la_{\al}(\cdot)\in\LL^2$ such that
 \begin{equation}\label{eq1112a}
 	D_uJ_{\alpha}\big(u_{\alpha}(\cdot),\la_\al(\cdot)\big) =0\mbox{ and }D_\la J_{\alpha}\big(u_{\alpha}(\cdot),\la_\al(\cdot)\big) =0,
 \end{equation}
where
$D_uJ_{\alpha}\big(u(\cdot),\la(\cdot)\big)\in\cU$ is the partial derivative of $J_{\alpha}$ with respect to $u(\cdot)$, namely,
\[\<D_uJ_{\alpha}\big(u(\cdot),\la(\cdot)\big),v(\cdot)\>_\cU
\equiv\lim_{\ep\to 0}\ep^{-1}\(J_\al(u(\cdot)+\ep v(\cdot),\la(\cdot))-J_\al(u(\cdot),\la(\cdot))\).\]
$D_\la J_{\alpha}\big(u(\cdot),\la(\cdot)\big)$ is defined similarly.

Problem 1 will be solved in two steps.   Fixing $\la(\cdot)\in \LL^2$ and denoting
 $J_{\alpha}\big(u(\cdot),\la(\cdot)\big)$ by $J_{\alpha,\la}\big(u(\cdot)\big)$, our first step is to  find $u_{\alpha, \lambda} (\cdot) $ satisfying
  $D J_{\alpha,\la}\big(u_{\al,\la}(\cdot)\big)=0$. 
 Once we have found the control $u_{\alpha,\la}(\cdot)$,  the second step is to search for the ELM $\lambda_{\alpha}(\cdot)$ to 
 satisfy the second part of (\ref{eq1112a}), namely, $\EE X^{u_{\alpha,\la}}(\cdot)=\al(\cdot)$. 
 
 The following Lemmas \ref{lemma3.3} and    \ref{the3.4} are key conclusions of Problem 1.
 
 \begin{lemma}\label{lemma3.3}
	Suppose that (H1) and (H2)  hold. Then, for $x \in  \RR^n$ and $\al(\cdot)\in\LL^2_0,$ $\la(\cdot)\in\LL^2$ fixed, there exists a unique 
	$u_{\al,\la}(\cdot) \in \mathcal{U}$ such that $D J_{\alpha,\la}\big(u_{\alpha,\la}(\cdot)\big) =0$. Furthermore, $u_{\al,\la}(\cdot)$ 
	is determined by  the  stationary condition
	\begin{equation}\label{stac1}
		R(s)u (s)+B(s)^{\top}Y(s)+D(s)^{\top}Z(s) = 0, \; a.e.\; s \in [0, T], \; a.s.,
	\end{equation}
	where $(X_{\al,\la}(\cdot), Y_{\al,\la}(\cdot), Z_{\al,\la} (\cdot))$ is the solution to the following 
	FBSDE: for $s \in [0, T] $,
	\begin{equation}\label{FBSDE1}
		\begin{cases}
			dX(s)=[AX+Bu+A_1\alpha]ds+[CX+Du+C_1\alpha]dW(s), \\
			dY(s)\;=- [ A^{\top}Y+C^{\top}Z+QX + \lambda] ds+Z(s)dW(s), \\
			X(0)\; \;=x, \quad Y(T)=GX(T).
		\end{cases}
	\end{equation}
	\end{lemma}

Note that in (\ref{stac1}) and (\ref{FBSDE1}) above, we have omitted the dependence of the processes
 $(u_{\al,\la}(\cdot),X_{\al,\la}(\cdot), Y_{\al,\la}(\cdot), Z_{\al,\la} (\cdot))$ on the fixed processes $\al(\cdot)$ and $\la(\cdot)$. We will continue to omit subscripts of such processes in equations when there is no confusion.
Substituting (\ref{stac1})  into (\ref{FBSDE1}),  $(X_{\al,\la}(\cdot), Y_{\al,\la}(\cdot), Z_{\al,\la} (\cdot))$ satisfies the following FBSDE:
 for $s\in [0, T]$, 
\begin{equation}\label{linear1}
	\left\{
	\begin{aligned}
		dX(s)=&\; \{AX- BR^{-1}[B^{\top}Y+D^{\top}Z]+A_1\alpha\}ds \\
		&+ \{CX- DR^{-1}[B^{\top}Y+D^{\top}Z]+C_1\alpha\}dW(s),\\
		dY(s)=&- [ A^{\top}Y+C^{\top}Z+QX + \lambda  ] ds + Z(s)dW(s),\\
		X(0)= &\; x, \; Y(T)=GX(T).
	\end{aligned}
	\right.
\end{equation}

\begin{lemma}\label{the3.4}
	Let (H1) and (H2)  hold. Then, for $x \in  \RR^n$ and $\al(\cdot)\in\LL^2_0,$ $\la(\cdot)\in\LL^2$  fixed, the coupled system (\ref{linear1})
	 admits a unique adapted solution
	  $\big(X_{\al,\la}(\cdot), Y_{\al,\la}(\cdot), Z_{\al,\la}(\cdot)\big) \in \big(L^{2,c}_\FF(\RR^n)\big)^2 \times L_{\FF}^{2}(\mathbb{R}^n )$. 
	\end{lemma}
	
By the uniqueness above, we can write the solution to (\ref{linear1}) as linear combinations of $x,\;\la(\cdot),\;\al(\cdot)$. Thus, there are three linear operators $\cP$, $\cL_1$ and $\cL_2$ taking values in $\LL^2$ from $\RR^n$, $\LL^2$ and $\LL^2_0$, respectively, such that
$\EE X_{\al,\la}(\cdot)=\cP x(\cdot)+\cL_1\la(\cdot)+\cL_2\al(\cdot)$. Thus,
	\begin{equation}  \label{xu5} 
		\alpha(t) = (\mathcal{P}x)(t)  + (\mathcal{L}_1\lambda)(t)+(\mathcal{L}_2 \alpha)(t). 
\end{equation}

{\color{blue}\begin{remark}\label{rem0523a}
We will decouple the system (\ref{linear1}) in Section \ref{sec4} and obtain the equation (\ref{xu3}) for $X_{\alpha, \lambda}$ alone, together with a BSDE 
for an auxiliary process pair which can be solved separately. After that, (\ref{xu3}) can be  solved by standard
SDE technique. Thus,  $X_{\alpha, \lambda}$ can be explicitly represented as a linear combination of  $x$, $\lambda$ and $\alpha$. Taking expectation then
gives the linear operators $\cP$, $\cL_1$ and $\cL_2$. 
\end{remark}}

	 \begin{remark}
	 	Under an additional condition, we can prove that $\cL_1$ is invertible. 
	 	In this case, we can uniqely solve $\lambda(\cdot)$ from (\ref{xu5}) in terms of $\al(\cdot)$, 
	 	and then  Problem  (MFSLQ) will be converted to an SLQ control problem with deterministic control variable $\al(\cdot)$.
	 	However, as we will see from the next lemma,  the uniqueness of $\la(\cdot)$ is not important. 
	 	\end{remark}
	 	
	\begin{lemma}\label{lem1115a}
For any $\la(\cdot)\in\LL^2$ such that $\EE X_{\al,\la}(\cdot)=\al(\cdot)$, we have $u_{\al,\la}(\cdot)=u^*_\al(\cdot)$.
\end{lemma}

After we obtain  the optimal $u^*_\al(\cdot)$ for Problem 1, we finally  consider the problem with state equation (\ref{OState}) (with $u(\cdot)$ there replaced by $u^*_\al(\cdot)$),   and objective functional
\begin{equation}\label{eq1116a}
	\begin{aligned}
		\hat{J}_\al(u^*_\al(\cdot)) =\mathbb{E} \Big\{ \int_{0}^{T}\Big(\langle QX,X\rangle+\langle Q_1\alpha,\alpha\rangle 
			+ \langle Ru^*_\al , u^*_\al \rangle   \Big)ds
			+	\langle GX(T),X(T)\rangle 	\Big\}.
\end{aligned}\			
	\end{equation}

{\bf Problem 2}: Find $\al^*(\cdot)\in  \LL^2_0$ such that
\[\inf_{u(\cdot) \in \mathcal{U}}J\big( u(\cdot)\big) =  \hat{J}_{\al^*}(u^*_{\al^*}(\cdot)) =   \inf_{\alpha(\cdot) \in \LL^2_0 } \hat{J}_\al(u^*_\al(\cdot)) .\]

{\color{blue}\begin{lemma}\label{lem0523a}
The unique optimal control   $u^*_\al(\cdot)$ for Problem 1  can be 
written as
\begin{equation}  \label{xiong} 
	u^*_\al(\cdot)  = (\mathcal{P}_2x)(\cdot)   +(\mathcal{K}_2 \alpha)(\cdot).
\end{equation}
Further,  the state process for Problem 2  can be represented as 
\begin{equation}  \label{xiong2} 
	X(\cdot)  = (\mathcal{P}_1x)(\cdot)   +(\mathcal{K}_1 \alpha)(\cdot),
\end{equation}
and 
\begin{equation}  \label{xiong3} 
	X(T)  = \mathcal{P}_3x  +\mathcal{K}_3 \alpha,
\end{equation}
where  $\mathcal{K}_i$ and 	$\mathcal{P}_i$ are some linear operators, $i= 1, 2, 3$, on suitable spaces.
\end{lemma}}

\begin{lemma}\label{lem1116b}
Let (H1) and (H2) hold. Then  $\al^{\ast}(\cdot)$ is optimal for Problem 2 if and only if 
	\begin{equation}\label{alsolution}
 (\mathcal{K}^*_1 Q  \mathcal{K}_1 + \EE Q_1 + \mathcal{K}^*_2R \mathcal{K}_2+ \mathcal{K}^*_3G \mathcal{K}_3) \al ^* + ( \mathcal{K}^*_1 Q  \mathcal{P}_1   +\mathcal{K}^*_2 R \mathcal{P}_2+\mathcal{K}^*_3 G \mathcal{P}_3  ) x=  0,
\end{equation}
where $\mathcal{K}^*_i$ are adjoint operators of $\mathcal{K}_i$ for $i = 1, 2 , 3$.
\end{lemma}

Finally, we summarize the results above to the main theorem of this paper.

\begin{theorem}\label{main0}
	Suppose that (H1) and (H2) hold. Then, the optimal control  $u^*(\cdot)$ of Problem (MFSLQ) is given by (\ref{stac1}) with $\big(X(\cdot),Y(\cdot),Z(\cdot)\big)$ satisfying (\ref{linear1}) (with $(\al(\cdot)$ there replaced by $\al^*(\cdot)$, and $\lambda(\cdot)$ there  replaced by any solution of  equation (\ref{xu5}) (with $(\al(\cdot)$ in equation (\ref{xu5}) replaced by $\al^*(\cdot) )$, 
	while the deterministic process $\al^*(\cdot)$ is any solution to (\ref{alsolution}).
\end{theorem}

\begin{remark}
	
The decoupling of the FBSDE (\ref{linear1}) will be discussed in Section \ref{sec4}. The homogeneous case (i.e. $\al=\la=0$) was established   by Tang \cite{tang2003}. For  equations with general non-homogeneous 
terms, this problem was studied by Kohlmann and Tang \cite{KT} using BMO-martingale
    type theory.  When the state process is one-dimensional with regime-switching,
    it was recently   investigated  by Hu $et\; al$ \cite{HSX} using BMO theory. In this article, we consider 
  this problem with a special class of nonhomogeneous terms without using BMO theory, which is very technical. 
  This special class is enough for our purpose in this paper.
  \end{remark}

\section{A preliminary stochastic maximum principle }\label{sec3}
\setcounter{equation}{0}
\renewcommand{\theequation}{\thesection.\arabic{equation}}

In this section, we start with the existence and uniqueness of the solution to the state equation (\ref{OState1}) and an adjoint BSDE (\ref{eBSDE}) using
contraction principle. Then, we prove Theorem \ref{SMP1} using convex variation principle and a general result for quadratic forms.

For the well-posedness of the state equation (\ref{OState1}), we have the following result.
\begin{theorem}\label{thm0515a}
Let (H1) hold. Then, for any initial state $x \in \RR^n$ and control $u(\cdot) \in \mathcal{U}$, the state equation  (\ref{OState1}) admits a unique adapted solution
$
X(\cdot)\in L^{2,c}_\FF(\RR^n),
$
and there exists a constant $K >0$, which is independent of $x$ and $u(\cdot)$, such that
\begin{equation}\label{statees}
\mathbb{E}\left[\sup_{0\leq s \leq T}|X(s)|^{2}\right] \leq K \left[x^{2 }+\mathbb{E} \int^{T}_{0}|u(s)|^{2}ds\right].
\end{equation}
\end{theorem}
\begin{proof} 
 For any $u(\cdot) \in \mathcal{U}$ fixed, we define an operator $\Gamma$ on $L^{2,c}_\FF(\RR^n)$:
\begin{equation*}
\begin{aligned}
(\Gamma \widetilde{X})(t) := &\, \, x+ \int^{t}_{0} [A\widetilde{X}+A_1\mathbb{E}\widetilde{X}+Bu]ds+\int^{t}_{0}[C\widetilde{X}+C_1\mathbb{E}\widetilde{X}+Du]dW(s).
\end{aligned}
\end{equation*}
Applying Cauchy–Schwarz and Burkholder–Davis–Gundy (BDG for short) inequalities, we obtain
\begin{eqnarray}\label{estima2}
\mathbb{E}\Big[\sup_{0 \leq t \leq \tau} |( \Gamma \widetilde{X})(t)|^{2}\Big] &\leq & \, K \mathbb{E}\Big[ x^{2}+ \Big(\int^{\tau}_{0} |A\widetilde{X}|ds\Big)^{2}+  
\Big(\int^{\tau}_{0}|A_1\mathbb{E}\widetilde{X}|ds\Big)^{2}+\Big(\int^{\tau}_{0}|Bu|ds\Big)^{2} \nonumber\\
&&+\int^{\tau}_{0}|C\widetilde{X}|^{2}ds+\int^{\tau}_{0}|C_1\mathbb{E}\widetilde{X}|^{2}ds +\int^{\tau}_{0} |Du|
^{2}ds \Big]\nonumber\\
&\leq &
 K \mathbb{E}\Big[  x^{2}+K\tau\sup_{0 \leq t \leq \tau} |  \widetilde{X}(t)|^{2}+K\(1+\tau\)\int^{\tau}_{0}|u(s)|^{2}ds\Big].
\end{eqnarray}
Hence, $(\Gamma \widetilde{X})(t) \in  L^{2,c}_\FF(\RR^n)$. For $\widetilde{X}_{1}(\cdot) $, $\widetilde{X}_{2}(\cdot) \in  L^{2,c}_\FF(\RR^n)  $, we have 
\begin{eqnarray*}
&&\mathbb{E}\Big[\sup_{0 \leq t \leq \tau} |\Gamma   (\widetilde{X}_{1}-\widetilde{X}_{2})(t)|^{2}\Big] \\
&\leq &    K \mathbb{E} \Big[\Big(\int^{\tau}_{0}|A(s)(\widetilde{X}_{1}-\widetilde{X}_{2})(s)|ds\Big)^{2}+\Big(\int^{\tau}_{0}|A_{1}(s)\mathbb{E}(\widetilde{X}_{1}-\widetilde{X}_{2})(s)|ds\Big)^{2}\\
&&+\int^{\tau}_{0}|C(s)(\widetilde{X}_{1}-\widetilde{X}_{2})(s)|^{2}ds+\int^{\tau}_{0}|C_{1}(s)\mathbb{E}(\widetilde{X}_{1}-\widetilde{X}_{2})(s)|^{2}ds\Big] \\
&\leq&  K\tau \mathbb{E} \Big[\sup_{0 \leq t \leq \tau}|\widetilde{X}_{1}-\widetilde{X}_{2}|^{2}(t) \Big].
\end{eqnarray*}
Taking $\tau > 0$ small enough, by  contraction mapping theorem, there exists a unique strong solution on $[0, \tau ]$. Moreover,  (\ref{estima2}) implies that this unique solution satisfies estimation (\ref{statees}). Then we can apply the usual continuation argument to get the unique adapted solution on $[0, T]$, and  (\ref{statees}) follows from BDG and Gronwall inequalities easily.
\end{proof}

Next, we consider the following BSDE with random coefficients and mean-field terms: for $s \in [0, T]$,
\begin{equation}\label{eBSDE}
\left\{
\begin{aligned}
dY(s) &=  -\{A^{\top}Y+C^{\top}Z+\mathbb{E}[A_{1}^{\top}Y+C_{1}^{\top}Z]  +q\}ds+ Z(s)dW(s),\\
Y(T)&=\zeta\in L^{2}_{\cF_T}(\RR^n).
\end{aligned}
\right.
\end{equation}
This equation is of the same form as the second equation of (\ref{FBSDEoptimal}) with $QX^*+\EE Q_1\EE X^*$ there replaced by $q$ here, which is regarded as known after the first equation of (\ref{FBSDEoptimal}) is solved. Also, 
$\zeta$ here replaces $GX^*(T)$ there.

\begin{theorem}\label{thm2}
Let  (H1) and  (H2) hold. Then, BSDE (\ref{eBSDE}) has a unique solution 
$$\big(Y(\cdot), Z(\cdot)\big) \in L^{2,c}_\FF(\RR^n) \times L_{\FF}^{2}(\mathbb{R}^n ).  $$
Moreover, 
\begin{equation}\label{eq1006a}
\mathbb{E}\left [\sup_{0 \leq t \leq T}|Y(t)|^{2}\right]+ \mathbb{E} \int^{T}_{0} |Z(s)|^2ds \leq K\mathbb{E}\left[ |\zeta|^{2}+\int^{T}_{0}|q(s)|^{2}ds\right].
\end{equation}
\end{theorem}
\begin{proof}
Define a norm $||X(\cdot)||_{\sigma}$ on $  L_{\FF}^{2} (\mathbb{R}^{n})$ as
$$ ||X(\cdot)||_{\sigma} =\(\mathbb{E}\int^{T}_{0}e^{\sigma t}|X(t)|^{2}dt\)^{\frac{1}{2}},$$
where $\sigma>0$ will be determined later.   It is easy to verify that  for any $\big(y(\cdot), z(\cdot)\big) \in \big(L_{\FF}^{2} (\mathbb{R}^{n  }) \big)^2$, the following BSDE has a unique  adapted solution $\big(Y(\cdot), Z(\cdot)\big) \in \big(L_{\FF}^{2} (\mathbb{R}^{n}) \big)^2$: 
\begin{equation*}
\left\{
\begin{array}{ccl}
dY(s) &= & -\{A^{\top}Y+C^{\top}Z+\mathbb{E}[A^{\top}_{1}y+C^{\top}_{1}z]
+q\}ds
+ ZdW(s),\\
Y(T)&=&\zeta, \; s \in [0, T].
\end{array}
\right.
\end{equation*}
Define the operator $\Gamma $ as $\Gamma (y, z) := (Y, Z)$. For $(y_{i}, z_{i}) \in   \big(L_{\FF}^{2} (\mathbb{R}^{n})\big)^2$, $i=1,\ 2$, we denote $ (Y_{i}, Z_{i}) = \Gamma (y_{i}, z_{i})$. Set $$(\widetilde{Y}, \widetilde{Z}) = (Y_{1}-Y_{2}, Z_{1}-Z_{2})\mbox{ and }(\widetilde{y}, \widetilde{z}) = (y_{1}-y_{2}, z_{1}-z_{2}).$$ Applying It\^o's formula to $e^{\sigma t} |\widetilde{Y}(t)|^{2}$, we obtain
\begin{equation*}
\begin{aligned}
d \big(e^{\sigma t} |\widetilde{Y}(t)|^{2}\big) =& \sigma e^{\sigma t} |\widetilde{Y}|^{2}dt -2e^{\sigma t} \langle\widetilde{Y}, A^{\top} \widetilde{Y}+C^{\top}\widetilde{Z}+\mathbb{E}[A^{\top}_{1}\widetilde{y}+C^{\top}_{1}\widetilde{z}] \rangle dt \\
& + 2e^{\sigma t}\langle\widetilde{Y}, \widetilde{Z}\rangle dW(t) +e^{\sigma t} |\widetilde{Z}|^{2}dt.
\end{aligned}
\end{equation*}
Therefore, 
\begin{equation*}
\begin{aligned}
|\widetilde{Y}(0)|^{2} &+ \sigma \mathbb{E}\left[\int^{T}_{0} e^{\sigma s} |\widetilde{Y}|^{2}ds\right]+ \mathbb{E}\left[\int^{T}_{0} e^{\sigma s} |\widetilde{Z}|^{2}ds \right]\\
=& \; 2\mathbb{E}\left[\int^{T}_{0} e^{\sigma s} \langle \widetilde{Y}, A^{\top} \widetilde{Y}+C^{\top}\widetilde{Z}+\mathbb{E}[A^{\top}_{1}\widetilde{y}]+\mathbb{E}[C^{\top}_{1}\widetilde{z}] \rangle ds\right]\\
\leq& \; 2K \mathbb{E}\left\{\int^{T}_{0}e^{\sigma s}\(|\widetilde{Y}|^{2}+|\widetilde{Z}^{\top}\widetilde{Y}|+ | \mathbb{E}(\widetilde{y}^{\top}) \widetilde{Y}| + | \mathbb{E}(\widetilde{z}^{\top}) \widetilde{Y}|\)ds \right\} \\
\leq &\;    \mathbb{E}\bigg\{\int^{T}_{0} e^{\sigma s} \Big(2K |\widetilde{Y}|^{2}+ \frac{1}{2}|\widetilde{Z}|^{2} + 8K^{2} |\widetilde{Y}|^{2} +\frac{\sigma}{4}|\widetilde{Y}|^{2} + \frac{16K^{2}}{\sigma} \mathbb{E}|\widetilde{y}|^{2}\\
&\; +\frac{\sigma}{4}|\widetilde{Y}|^{2} +\frac{16K^{2}}{\sigma} \mathbb{E}|\widetilde{z}|^{2} \Big)ds\bigg\}  \\
=&\;   \mathbb{E}\left\{\int^{T}_{0}  e^{\sigma s} \left[\(2K+8K^{2} + \frac{\sigma}{2}\) |\widetilde{Y}|^{2}+ \frac{1}{2}|\widetilde{Z}|^{2} + \frac{16K^{2}}{\sigma}(  \mathbb{E}|\widetilde{z}|^{2}+  \mathbb{E}|\widetilde{y}|^{2})\right]ds \right \}.
\end{aligned}
\end{equation*}
We then have,
\begin{eqnarray*}
&&\(\frac{\sigma}{2}-2K-8K^{2}\)   \mathbb{E}\left\{\int^{T}_{0} e^{\sigma s}  |\widetilde{Y}|^{2}ds \right\}+ \frac{1}{2} \mathbb{E}\left\{\int^{T}_{0} e^{\sigma s}  |\widetilde{Z}|^{2}ds \right\}\\
&\leq& \frac{16K^{2}}{\sigma}   \mathbb{E}\left\{\int^{T}_{0} e^{\sigma s}  (|\widetilde{z}|^{2}+  |\widetilde{y}|^{2})ds\right  \}.
\end{eqnarray*}
Thus, 
$\|(\widetilde{Y},\widetilde{Z})\|^2_\si\le L\|(\widetilde{y},\widetilde{z})\|^2_\si$, 
where $L=\max\(\frac{32K^2}{\si(\si-4K-16K^2)},\frac{32K^2}{\si}\)$ which is clearly less than 1 if we take
$\sigma= 32K^{2}+4K+1$. In this case, $\Gamma$ is a contraction mapping.  Therefore, BSDE (\ref{eBSDE}) has a unique solution $(Y(\cdot), Z(\cdot)) \in \(L_{\FF}^{2} (\mathbb{R}^{n} ) \)^2$. Moreover, by the properties of stochastic integral in equation (\ref{eBSDE}), it is clear $Y(\cdot)$ is continuous. The second moment of the sup norm of $Y(\cdot)$ can be estimated by BDG inequality. Thus,
 $Y(\cdot) \in L^{2,c}_\FF(\RR^n). $ Then, the existence and uniqueness of solution of BSDE (\ref{eBSDE})  are obtained. 

The proof of (\ref{eq1006a}) follows from the same arguments as those for (2.11) in (\cite{Yong1999}, p.349).
\end{proof}

The next lemma will be useful in proving the optimaility for quadratic functionals. Let $\HH$ be a Hilbert space and $\cQ$  a self-adjoint operator
on $\HH$. We say that $\cQ>0$ (resp. $\cQ\ge 0$) if for any $h\in\HH\setminus \{0\}$, we have $\<\cQ h,h\>_\HH>0$ (resp. $\ge 0$). We assume 
$\cQ\ge 0$ and
\begin{equation}\label{eq1114a}
F(h)=\<\cQ h,h\>_\HH-2\<h,b \>_\HH+c,\end{equation}
where $b\in\HH$ and $c\in\RR$. We say that the quadratic form $F$ is positive definite (resp. semi-definite) if $\cQ>0$ (resp. $\cQ\ge 0$).

\begin{lemma}\label{lem1114a}
i) If $\cQ>0$, then there exists a unique $h_0\in\HH$ such that $F$ attains its infimum.

ii) If $\cQ\ge 0$, then $h_0$ attains the infimum of $F$ if and only if $D F(h_0)=0$.
\end{lemma}
\begin{proof}
i) Note that $\cQ$ is invertible. Then
\begin{eqnarray*}
F(h)&=&\left|\cQ^{1/2}h\right|^2_{\HH}-2\<\cQ^{1/2}h,\cQ^{-1/2}b\>_\HH+c\\
&=&\left|\cQ^{1/2}h-\cQ^{-1/2}b\right|^2_\HH+c-\left|\cQ^{-1/2}b\right|^2_\HH,
\end{eqnarray*}
attains its infimum at the only point $h=\cQ^{-1}b$.

ii) If $h_0$ attains the infimum of $F$, then for any $h\in\HH$, we have
\begin{eqnarray*}
0&\le&\lim_{\ep\to 0}\ep^{-1}\(F(h_0+\ep h)-F(h_0)\)\\
&=&\lim_{\ep\to 0}\ep^{-1}\(\<\cQ (h_0+\ep h),h_0+\ep h\>_\HH-2\<h_0+\ep h,b \>_\HH-\<\cQ h_0,h_0\>_\HH+2\<h_0,b \>_\HH\)\\
&=&2\<\cQ h_0-b,h\>.
\end{eqnarray*}
Thus, $D F(h_0)=2(\cQ h_0-b)=0$.

On the other hand, suppose $D F(h_0)=0$. Then,  $\cQ h_0=b$. Then, for any $h\in\HH$, we have
\begin{eqnarray*}
F(h)-F(h_0)&=&\<\cQ h,h\>_\HH-2\<h,\cQ h_0 \>_\HH+c-\<\cQ h_0,h_0\>_\HH+2\<h_0,\cQ h_0 \>_\HH-c\\
&=&\<\cQ(h-h_0),h-h_0\>_\HH\ge 0.
\end{eqnarray*}
Thus, $h_0$ attains the infimum of $F$.
\end{proof}

Finally, we proceed to presenting the

{\em	Proof of Theorem \ref{SMP1}}: By Theorem \ref{thm0515a}, $X(\cdot)$ can be written as a linear combination of
$x\in\RR^n$ and $u(\cdot)\in\cU$. Therefore, $J$ can be written as a quadratic functional
 of $u(\cdot)$ with a finite infimum. Thus, $J$ has the 
form of (\ref{eq1114a}) with $\HH=\cU$ and 
$\<\cQ u,u\>_\cU\ge\de|u|^2_{\cU}>0$. By Lemma \ref{lem1114a} i), $J$ attains its infimum at a unique control $u^*(\cdot)$,
 and hence,
Problem (MFSLQ) has a unique optimal control $ u^*(\cdot)$. Further, $u^*(\cdot)$ is determined by $D J(u^*(\cdot))=0$.

Let $X^\ast(\cdot)$ be the optimal process that satisfies equation (\ref{OState1}). Let $X^{\epsilon}(\cdot)$ denote the state trajectory with respect to the control $u^{\epsilon}(\cdot)=u^\ast(\cdot)+\epsilon v (\cdot)$, where  $\epsilon \in\RR$,  $v (\cdot) \in \mathcal{U} $. 

We introduce the variation equation:
\begin{equation*}
	\left\{
\begin{aligned}
		dX_{1}(s)&=[AX_{1}+A_1\mathbb{E}X_{1}+Bv]ds+[CX_{1}+C_1\mathbb{E}X_{1}+Dv]dW(s), \; s \in [0, T], \\
		X_{1}(0) &=0.
\end{aligned}
\right.
\end{equation*}
By linearity, it is easy to verify that
$$
	X^{\epsilon}(s)-X^\ast(s)= \epsilon X_1(s), \qquad  \forall s \in [0, T],\ a.s..
	$$
Then, there exists  a constant $K>0$ such that for $\epsilon \in [0, 1]$
	$$
	\sup_{0 \leq s \leq T} \mathbb{E} |X^{\epsilon}(s) -X^\ast(s)| ^2 \leq K \epsilon^2.
	$$

Note that
\begin{eqnarray}\label{eq0921a}
&&J\(u^\ep (\cdot)\)-J\(u^*(\cdot)\)\\
&=&2\ep \mathbb{E}\left\{\int_0^T\Big( \langle QX^\ast,X_{1}\rangle+\langle Q_1\mathbb{E}X^\ast,\mathbb{E}X_{1}\rangle+ \langle Ru^\ast,v \rangle \Big)ds 
		+\langle GX^\ast(T), X_{1}(T) \rangle\right\}+\ep^2I\nonumber\\
&=&2\ep \mathbb{E}\left\{\int_0^T  \Big(\langle QX^\ast+
		\mathbb{E}Q_1\mathbb{E}X^\ast, X_{1}\rangle+\langle Ru^\ast, v \rangle \Big) ds + \langle GX^\ast(T), X_1(T)\rangle \right\}+\ep^2I,\nonumber
\end{eqnarray}
where 
\[I=\EE\(\int^T_0\(\<QX_1,X_1\>+\<Q_1\EE X_1,\EE X_1\>+\<Rv,v\>\)ds+\<GX_1(T),X_1(T)\>\)\ge 0.\]
Applying It\^o's formula to $\langle Y^\ast(\cdot), X_{1} (\cdot)\rangle $ and taking expectation, we obtain
	\begin{equation}\label{25q}
		\begin{aligned}
			\mathbb{E}\langle GX^\ast(T), X_{1}(T)\rangle
			=\mathbb{E}\left\{\int_0^T \Big( \langle -QX^\ast-\mathbb{E}Q_1\mathbb{E}X^\ast,  X_1\rangle+   \langle B^\top Y^\ast +D^\top Z^\ast,  v \rangle \Big) ds
			\right\}.\\
		\end{aligned}
	\end{equation}
Combining (\ref{eq0921a}) and (\ref{25q}), we get
\begin{equation}\label{eq0921b}
J\(u^\ep (\cdot)\)-J\(u^*(\cdot)\) =2\ep\mathbb{E} \int^T_0 \langle Ru^\ast+B^{\top}Y^\ast+D^{\top}Z^\ast, v \rangle ds+\ep^2I.\end{equation}
Therefore,
\[0= D J(u^*(\cdot))= 2(Ru^\ast+B^{\top}Y^\ast+D^{\top}Z^\ast).\]
	This finishes the proof of (\ref{eq0921c}).
	\qed
	
	\begin{remark}
		Note that it is difficult  to decouple  the optimality system    (\ref{FBSDEoptimal}, \ref{eq0921c}) due to the term $\mathbb{E}[A^\top_1Y^\ast +C^\top_1Z^\ast]$  in the adjoint equation. However, by  virtual of the ELM method in the next section, we will find a procedure to obtain the optimal control.
				\end{remark}

 \section{Extended LaGrange multiplier method}\label{sec4}
 \setcounter{equation}{0}
\renewcommand{\theequation}{\thesection.\arabic{equation}}

In this section, we consider Problem 1. First, we present the proofs of Lemmas \ref{lem1114b}-\ref{the3.4} and \ref{lem1115a}. Then, we will decouple the 
FBSDE (\ref{linear1}), and give a  representation of the optimal control $u^*_{\alpha}(\cdot)$ for Problem 1. 

{\em Proof of Lemma \ref{lem1114b}}: Similar to Theorem \ref{thm0515a}, for $x\in\RR^n$, $\al(\cdot)\in\LL^2_0$ and $u(\cdot)\in\cU$ fixed, SDE (\ref{OState})
 has a unique solution. Hence, $X(\cdot)$ can be written as a linear combination of $x$, $\al(\cdot)$ and $u(\cdot)$.
  Therefore, $\hat{J}_\al$ can be written as a quadratic form of $u(\cdot)$ with $\HH=\cU$ when $x$ and $\al(\cdot)$ are fixed. Further, $\EE X^u(\cdot)=\al(\cdot)$ can be written as $\cR u(\cdot)=h_1$, where $\cR$ is a linear operator from $\HH$ to $\LL^2$. Note that $h_1$ is in the range of the operator $\cR$.
   Thus, $h_1=\cR h_2$ for $h_2\in\HH$ and $\cR(u(\cdot)-h_2)=0$. Therefore, $u(\cdot)=h_2+h$ for $h\in\HH_1$ 
   which is the null space of $\cR$. So, $G(h)\equiv \hat{J}_\al(h_2+h)$ is a quadratic form on $\HH_1$ with the same quadratic operator $\cQ$ (and 
   modified $b$ and $c$). This implies the existence of a unique optimal control $u^*_\al(\cdot)$ for the constrained problem.
   \qed

{\em Proof of Lemma \ref{lemma3.3}}: Similar to the proof of Lemma \ref{lem1114b} above, we get the existence and uniqueness of such $u_{\al,\la}(\cdot)$ since $J_{\al,\la}(u(\cdot))$ is a positive definite quadratic form. 
Further, $u_{\al,\la}(\cdot)$ attains the infimum of $J_{\al,\la}$.

Let $X_{\al,\la}(\cdot)$ be  the solution of $(\ref{OState})$ corresponding  to the control $u_{\al,\la}(\cdot)$.  We will drop the subscript  $(\al, \la)$ in the calculations  when there is no confusion.
 Let $X_{\epsilon}(\cdot)$ denote the state trajectory with respect to the control $u_{\epsilon}(\cdot)=u(\cdot)+\epsilon v(\cdot)  $, where $\epsilon \in (0, 1]$ and  $v(\cdot) \in \mathcal{U} $. 
	
	Next, we introduce the variation equation:
	\begin{equation}\label{VAequation}
		\left\{
		\begin{aligned}
			dX_{1}(s)=&\; [A(s)X_{1}(s)+B(s)v(s)]ds+[C(s)X_{1}(s)+D(s)v(s)]dW(s), \; s \in [0, T], \\
			X_{1}(0) =&\; 0.
		\end{aligned}
		\right.
	\end{equation}
Note that
	\begin{align*}
		 &\; \frac{	J_{\alpha,\la}\(u_{\epsilon}(\cdot)\)-	J_{\alpha,\la}\(u(\cdot)\)}{\epsilon} \\
		=&\; \mathbb{E}\Big\{\frac{1}{\epsilon}\Big[\int_0^T\Big(\langle QX_{\epsilon},X_{\epsilon}\rangle-\langle QX,X\rangle 
		+\epsilon^2\langle Rv,v \rangle +2\epsilon \langle Ru, v \rangle  + \langle 2 \lambda, X_{\epsilon}- X\rangle \Big)ds\\
		&\quad +\langle GX_{\epsilon}(T),X_{\epsilon}(T)\rangle -\langle GX(T),X(T)\rangle \Big]\Big\}.
	\end{align*}
	Taking $\epsilon \to 0$, we have 
	\begin{equation}\label{28}
	\begin{aligned}
			0 = &\;  \lim_{\epsilon \to 0} \frac{	J_{\alpha,\la}\(u_{\epsilon}(\cdot)\)-	J_{\alpha,\la}\(u(\cdot)\)}{\epsilon} \\
		= &\; 2\mathbb{E}\Big\{\int_0^T\Big( \langle QX,X_1\rangle 
		+ \langle Ru, v \rangle  +  \langle  \lambda , X_1 \rangle \Big) ds+
		\langle GX(T),X_1(T)\rangle  \Big\}.
	\end{aligned}
	\end{equation}
	Applying It\^o's formula to $\langle Y(\cdot), X_{1} (\cdot)\rangle $, we obtain
	\begin{align*}
		d\langle Y(s), X_{1}(s)\rangle &=\; \langle  -( A^{\top}Y+C^{\top}Z+QX + \lambda), X_{1} \rangle ds +\langle Z,  X_{1}\rangle dW(s)\\
		&\quad +\langle Y, AX_{1}+Bv \rangle ds 
		+\langle Y, CX_{1}+Dv \rangle dW(s)
		+\langle Z, CX_{1}+Dv \rangle ds. 
	\end{align*}
	Therefore
	\begin{equation}\label{29}
		\begin{aligned}
			\mathbb{E}\langle GX(T), X_{1}(T)\rangle 
			=&\; \mathbb{E}\Big\{\int_0^T \Big( \langle -(QX+\lambda), X_{1}\rangle+\langle B^{\top}Y+D^{\top}Z, v\rangle \Big) ds\Big\}.
		\end{aligned}
	\end{equation}
	Substituting  (\ref{29})  into (\ref{28}), we obtain that  
	\begin{align}\label{eq1004a}
	0 =	&\; \lim_{\epsilon \to 0} \frac{	J_{\alpha,\la}\(u_{\epsilon}(\cdot)\)-	J_{\alpha,\la}\(u(\cdot)\)}{\epsilon} = \; 2 \mathbb{E}\int_0^T 
		 \langle Ru + B^{\top}Y+D^{\top}Z, v \rangle  ds.
	\end{align}
	Namely, (\ref{stac1}) holds.
	\qed

{\em Proof of Lemma \ref{the3.4}}: 
 Lemma \ref{lemma3.3} implies the existence and uniqueness 
of the solution to (\ref{linear1}). The $L^{2,c}_\FF$ and $L^2_\FF$-estimates follow from the same arguments as those in Theorems \ref{thm0515a}
and \ref{thm2}.
\qed

{\em Proof of Lemma \ref{lem1115a}}: As $u_{\al,\la}(\cdot)$ attains the infimum of $J_{\al,\la}(\cdot)$, we have
\[\hat{J}_\al(u_{\al,\la}(\cdot))+2\<\la,\EE X_{\al,\la}-\al\>_{\LL^2}
=J_{\al,\la}(u_{\al,\la}(\cdot))
\le J_{\al,\la}(u^*_\al)=\hat{J}_\al(u^*_\al(\cdot)).\]
Thus, for $\la$ such that $\EE X_{\al,\la}(\cdot)=\al(\cdot)$, we have $\hat{J}_\al(u_{\al,\la}(\cdot))\le \hat{J}_\al(u^*_\al(\cdot))$ which
becomes an equality due to the optimality of $u^*_\al(\cdot)$. The uniqueness from Lemma \ref{lem1114b} implies that
$u_{\al,\la}(\cdot)=u^*_\al(\cdot)$. 
\qed

{\color{blue}{\em Proof of Lemma \ref{lem0523a}}: Denote
\[\cD\equiv\left\{\cL^{-1}_1\be:\;\be\in\LL^2\right\}\subset 2^{\LL^2},\]
where $2^{\LL^2}$ is the set consisting of the subsets of $\LL^2$. Then, $\cD$ is a linear space by defining
\[c_1\cL^{-1}_1\be_1+c_2\cL^{-1}_1\be_2\equiv\{\la\in\LL^2:\; \la=c_1\la_1+c_2\la_2,\;\la_i\in\cL^{-1}_1\be_i,\ i=1,\ 2\},\]
with the convention that it is empty if $\cL^{-1}_1\be_1$ and/or $\cL^{-1}_1\be_2$ are empty.
By the linearity of $\cL_1$, it is easy to show that $c_1\cL^{-1}_1\be_1+c_2\cL^{-1}_1\be_2=\cL^{-1}_1(c_1\be_1+c_2\be_2)$. 

Similar to Remark \ref{rem0523a}, it follows from (\ref{stac1}) that $u_{\al,\la}(\cdot)$ can be written as a linear combination of $x,\ \al,\ \la$, namely
\[u_{\al,\la}(\cdot)=\ga_0 x(\cdot)+\ga_1\la (\cdot)+\ga_2\al(\cdot).\]
By Lemma \ref{lem1115a}, it is easy to see that $\ga_1$ is constant on each set $\cL^{-1}_1\be$. Namely, $\ga_1$ can be regarded as a linear operator on 
$\cD$. Thus,
\begin{eqnarray*}
u^*_\al(\cdot)&=&\ga_0 x(\cdot)+\ga_1\cL^{-1}_1(\al-\cP x-\cL_2\al) (\cdot)+\ga_2\al(\cdot)\\
&=&\(\ga_0-\ga_1\cL^{-1}_1\cP\)x(\cdot)+\(\ga_2+\ga_1\cL^{-1}_1(I-\cL_2)\)\al(\cdot)\\
&\equiv&(\mathcal{P}_2x)(\cdot)   +(\mathcal{K}_2 \alpha)(\cdot).
\end{eqnarray*}
This proves (\ref{xiong}). The other two equalities can be proved similarly.
\qed}

Since    FBSDE (\ref{linear1}) is fully coupled, we now use the invariant embedding skill to decouple it.  Trying an ansatz:
$$
Y(s) = \Sigma(s)X(s)+\phi(s), \; s \in [0,T],
$$
where 
$$
d\Sigma(s) = \Sigma_{1}(s)ds +\Psi(s)dW(s), \;   \Sigma(T) = G,
$$
and
$$
d\phi(s) = \phi_{1}(s)ds+\psi (s)dW(s),\; \phi(T)= 0, 
$$
with $\Sigma_1(\cdot)$ and $\phi_1(\cdot)$ being determined later.

Applying It\^o's formula to $Y(\cdot)$, we obtain
 \begin{equation}\label{combineY1}
	\begin{aligned}
		d Y(s) =& \;\Psi \Big(CX- DR^{-1}\(B^{\top}Y+D^{\top}Z\)+C_1\alpha \Big)  ds \\
		&\;+\Sigma \Big(AX- BR^{-1}\(B^{\top}Y+D^{\top}Z\)+A_1\alpha \Big) ds \\
		&\;+  \Sigma \Big(CX- DR^{-1}\(B^{\top}Y+D^{\top}Z\)+C_1\alpha \Big) dW(s) \\
		&\;+ \Sigma_{1} X ds+   \Psi  X dW(s) + \phi_{1}ds+\psi dW(s).
	\end{aligned}
\end{equation}
Combining (\ref{combineY1}) with the second equation of the system $(\ref{linear1})$,  we obtain
 \begin{equation}\label{c1}
	Z = \Sigma  \Big( CX- DR^{-1}\(B^{\top}Y+D^{\top}Z\)+C_1\alpha\Big)+\Psi X+ \psi 
\end{equation}
and
\begin{equation*}\label{c2}
	\begin{aligned}
		0=& \(A^{\top}Y+ C^{\top}Z+QX+\lambda \)+  \Psi  \Big( CX- DR^{-1}\(B^{\top}Y+D^{\top}Z\)+C_1 \alpha\Big)\\
		& +\Sigma \Big(AX- BR^{-1}\(B^{\top}Y+D^{\top}Z\)+A_1\alpha\Big)+ \Sigma_{1} X+\phi_{1}.
	\end{aligned}
\end{equation*}

Now we calculate $u(\cdot)$. Since
$$R u+B^{\top} Y+D^{\top}  Z = 0$$
and 
$$
Z= ( \Sigma C+ \Psi) X+ \Sigma  D u+ \Sigma  C_1 \alpha +\psi,
$$
we obtain
$$
R u+B^{\top}(\Sigma X+\phi)+D^{\top}(\Sigma C+\Psi) X+D^{\top} \Sigma D u+D^{\top} \Sigma C_1 \alpha+D^{\top} \psi=0.
$$
Namely
$$
\begin{aligned}
	\left(R+D^{\top} \Sigma D\right) u & =-\left\{\left(B^{\top} \Sigma+D^{\top} \Sigma C+D^{\top}\Psi) X+B^{\top} \phi+D^{\top} \Sigma C_1 \alpha+D^{\top} \psi\right\}\right. .\\
\end{aligned}
$$
As $R+D^{\top} \Sigma D$ is invertible,  we have
\begin{equation}\label{c3}
	\begin{aligned}
		u =-\left(R+D^{\top} \Sigma D)^{-1}\left\{\left(B^{\top} \Sigma+D^{\top} \Sigma C+D^{\top} \Psi\right) X
		+B^{\top} \phi+D^{\top} \psi+D^{\top} \Sigma C_1 \alpha\right\}\right..
	\end{aligned}
\end{equation}
Now we arrive at
$$
\begin{aligned}
	0  =&A^{\top}(\Sigma X+\phi)+C^{\top}(\Sigma C+\Psi) X+C^{\top} \Sigma D u+C^{\top} \Sigma C_1 \alpha+C^{\top} \psi  
	 +Q X+ \lambda\\
	&+\Psi\left(C X+D u+C_1 \alpha\right) 
	+\Sigma\left(A X+B u+A_1 \alpha\right)+\Sigma_1 X+\phi_1.
\end{aligned}
$$
Namely
$$
\begin{aligned}
	0 = & \left(A^{\top} \Sigma+C^{\top} \Sigma C+C^{\top} \Psi +Q+\Psi C+\Sigma A+\Sigma_1\right) X \\
	&+  \left(C^{\top} \Sigma D+\Psi D+\Sigma B\right) u+A^{\top} \phi+C^{\top} \Sigma C_1 \alpha+C^{\top} \psi \\
	& + \lambda+\Psi C_1 \alpha+\Sigma A_1 \alpha+\phi_1.
\end{aligned}
$$
Therefore,
$$
\begin{aligned}
	0	=& \Big(\Sigma_1 +A^{\top} \Sigma+\Sigma A+C^{\top} \Psi+\Psi C+C^{\top} \Sigma C  +Q \\
	&-\left(C^{\top} \Sigma D+\Psi D+\Sigma B\right)\left(R+D^{\top} \Sigma D\right)^{-1}\left(B^{\top} \Sigma+D^{\top} \Sigma C+D^\top \Psi\right) \Big) X \\
	& +\phi_1+A^{\top} \phi+C^{\top} \Sigma C_1 \alpha+C^{\top} \psi
	+ \lambda+\Psi C_1 \alpha+\Sigma A_1 \alpha \\
	&-(C^{\top} \Sigma D+\Psi  D+\Sigma B)\left(R+D^{\top} \Sigma D\right)^{-1}  (B^{\top} \phi+D^{\top} \Sigma C_1 \alpha
	+D^{\top} \psi).
\end{aligned}
$$
Now, we  define  $\Sigma_{1}(\cdot)$ and  $\phi_1(\cdot)$  as follows:
\begin{equation*}
	\begin{aligned}
		\Sigma_{1}=	&-\{\Sigma A+A^{\top}\Sigma+ \Psi C+ C^{\top}\Psi+C^{\top}\Sigma C +Q\\
		& -(\Sigma B+\Psi D+ C^{\top} \Sigma D)(D^{\top}\Sigma D+R)^{-1}(B^{\top}\Sigma + D^{\top}\Psi + D^{\top}\Sigma C)\}
	\end{aligned}
\end{equation*}
and 
\begin{equation*}
	\begin{aligned}
		\phi_1=& -\Bigg\{  \lambda + \Big(A^{\top}- (\Sigma B+\Psi D+ C^{\top} \Sigma D)(D^{\top}\Sigma D+R)^{-1}B^{\top}\Big) \phi \\
		&\; + \Big(C^{\top}- (\Sigma B+\Psi D+ C^{\top} \Sigma D)(D^{\top}\Sigma D+R)^{-1}D^{\top}\Big) \psi  \\
		&\;+ \Big( C^{\top}\Sigma C_1+  \Psi C_1 + \Sigma A_1 -(\Sigma B+\Psi D+ C^{\top} \Sigma D)(D^{\top}\Sigma D+R)^{-1}D^{\top} \Sigma C_1\Big) \alpha\Bigg\}.
	\end{aligned}
\end{equation*}
Namely, $\Sigma(\cdot)$ satisfies the stochastic Riccati equation: for $s\in [0,T]$,
\begin{equation}\label{R11}
	\left\{
	\begin{aligned}
		d\Sigma(s)=&-\{\Sigma A+A^{\top}\Sigma+ \Psi C+ C^{\top}\Psi+C^{\top}\Sigma C +Q\\
		& -(\Sigma B+\Psi D+ C^{\top} \Sigma D)(D^{\top}\Sigma D+R)^{-1}(B^{\top}\Sigma + D^{\top}\Psi + D^{\top}\Sigma C)\}ds+\Psi dW(s),\\
		\Sigma(T)=& G,
	\end{aligned}
	\right.
\end{equation}
and  $\phi(\cdot)$ satisfies the linear BSDE with unbounded coefficients:  for $s\in [0,T]$,
	\begin{equation} \label{base211}
		\left\{
		\begin{aligned}
			d \phi(s)&= -[\widehat{M} (s)\phi(s) + \widehat{N} (s)\psi (s)+ \lambda(s) + \widehat{Q}(s) \alpha(s)]ds + \psi (s) dW(s), \quad \\
			\phi(T) &= 0,
		\end{aligned}
		\right.
	\end{equation}
	where
		 $$
	 \widehat{M} = A^{\top}- (\Sigma B+\Psi D+ C^{\top} \Sigma D)(D^{\top}\Sigma D+R)^{-1}B^{\top}, 
	 $$	
	 $$
	 \widehat{N} = C^{\top}- (\Sigma B+\Psi D+ C^{\top} \Sigma D)(D^{\top}\Sigma D+R)^{-1}D^{\top},
	 $$
	 and
	  \[\widehat{Q}=C^{\top}\Sigma C_1+  \Psi C_1 + \Sigma A_1 -(\Sigma B+\Psi D+ C^{\top} \Sigma D)(D^{\top}\Sigma D+R)^{-1}D^{\top} \Sigma C_1.\]
Further, $X(\cdot)$ is governed by the following SDE: for $s\in [0,T]$,
\begin{equation}\label{xu3}
	\left\{
	\begin{aligned}
		dX(s) =&\; [\widehat{A}(s)X(s) +\widehat{B}(s)\phi(s)+\widehat{B}_{1}(s)\psi(s)+\widehat{A}_1(s)\al(s)]ds\\
		&\;+ [\widehat{C}(s)X(s)  + \widehat{D}(s)\phi(s)+\widehat{D}_{1}(s)\psi(s)+\widehat{C}_1(s)\al(s)]dW(s),\\
				X(0) =&\; x, 
	\end{aligned}
	\right.
\end{equation} 	
where
 \begin{equation*}
 	\begin{aligned}
 \widehat{A} = & \;\; A- B (D^{\top} \Sigma D+R )^{-1}(B^{\top}\Sigma+D^{\top} \Psi+D^{\top}\Sigma C ), \\
\widehat{A}_1= &\;\;   A_1 - B (D^{\top} \Sigma D+R )^{-1} D^{\top} \Sigma C_1, \\
\widehat{B} =& \;\;-B(D^{\top} \Sigma D+R )^{-1} B^\top,
\quad
\widehat{B}_1 =-B(D^{\top} \Sigma D+R )^{-1} D^\top,
\end{aligned}
\end{equation*}
 \begin{equation*}
	\begin{aligned}
\widehat{C} = & \;\; C- D (D^{\top} \Sigma D+R )^{-1}(B^{\top}\Sigma+D^{\top} \Psi+D^{\top}\Sigma C ), \\
\widehat{C}_1 =& \;\; C_1- D (D^{\top} \Sigma D+R )^{-1}D^{\top} \Sigma C_1, \\
\widehat{D} =& \;\; -D(D^{\top} \Sigma D+R )^{-1} B^\top,
\quad
\widehat{D}_1 =-D(D^{\top} \Sigma D+R )^{-1} D^\top.
\end{aligned}
\end{equation*}

\begin{lemma}\label{lem0510a}
	Let (H1) and (H2) hold.  Then the stochastic Riccati equation  (\ref{R11}) admits a unique adapted solution $\big(\Sigma(\cdot), \Psi(\cdot)\big) \in L^{\infty,c}_{\FF}(\mathbb{S}^{n}) \times L^{2}_\FF(\mathbb{S}^{n}) $
	such that for some $c >0$ 
	$$D^{\top}\Sigma D+R \geq c I_n, \; \quad a.e \; on \; [0, T], \; a.s. $$
	Moreover,    the linear BSDE (\ref{base211}) admits a unique  solution 
	$\(\phi (\cdot), \psi(\cdot)\) \in L^{2,c}_{\FF}(\mathbb{R}^n) \times L^{p,2}_{\FF}( \mathbb{R}^{n}) $, $\forall p\in (0,1)$.
\end{lemma}
\begin{proof}
The existence and unqueness of the solution to  the stochastic Riccati equation  (\ref{R11}) 
follows directly from Theorem 5.3 in Tang \cite{tang2003}. The conclusion 
$\big(\Sigma(\cdot), \Psi(\cdot)\big) \in L^{\infty,c}_{\FF}(\mathbb{S}^{n}) \times L^{2}_\FF(\mathbb{S}^{n}) $ follows from the same theorem.
Further, the estimation of Theorem 5.1 in \cite{tang2003} implies that $\Psi(\cdot)\in L^{p,2}_{\FF}(\RR^n)$, $\forall\ p>0$.
	Let $\big(X(\cdot), Y(\cdot), Z(\cdot)\big)$ be the unique solution of (\ref{linear1}). Define 
	$$
	\phi(s) = Y(s)-\Sigma(s)X(s),
	$$
	and
	$$
		\psi  = Z-\Bigg\{ \Sigma  \Big( CX- DR^{-1}\(B^{\top}Y+D^{\top}Z\)+C_1\alpha\Big)+\Psi X\Bigg\}.
		$$
		Applying It\^o's formula to $\phi(\cdot)$, it is easy to check that $\(\phi(\cdot), \psi(\cdot)\)$ is a solution of  (\ref{base211}). Also, we can check that $\big(\phi (\cdot), \psi(\cdot)\big) \in L^{2,c}_{\FF}(\mathbb{R}^n) \times L^{p,2}_{\FF}( \mathbb{R}^{n}) $, $\forall\ p\in(0,1)$ which follows from
\begin{eqnarray*}
		&&\EE\(\int^T_0|\Psi(s)X(s)|^2ds\)^{p}\\
		&\le&\EE\(\sup_{s\le T}|X(s)|^{2p}\(\int^T_0|\Psi(s)|^2ds\)^{p}\)\\
		&\le&\(\EE\(\sup_{s\le T}|X(s)|^2\)\)^{p}
		\(\EE\(\int^T_0|\Psi(s)|^2ds\)^{p/(1-p)}\)^{1-p}<\infty.
\end{eqnarray*}
		
		To prove the uniqueness of the solution to (\ref{base211}) and (\ref{xu3}), let $\(X(\cdot),\phi(\cdot),\psi(\cdot)\)$ in
		 $\(L^{2,c}_\FF(\RR^n)\)^2\times L^{p,2}_\FF(\RR^n)$ be a solution of (\ref{base211}) and (\ref{xu3}). Let $Y=\Sigma X+\phi$ and $$Z=(I+\Sigma DR^{-1}D^{\top})^{-1}\((\Sigma C+\Psi)X-\Sigma DR^{-1}Y+\Sigma C_1\al+\psi\).$$ Then, for any $p\in(0,1)$,  $\(X(\cdot),Y(\cdot),Z(\cdot) \)  \in\(L^{2,c}_\FF(\RR^n)\)^2\times L^{p,2}_\FF(\RR^n)$ is a solution to FBSDE (\ref{linear1}). 
		
		To prove that $Z(\cdot)\in L^2_\FF(\RR^n)$, we define
		\[\tau_k=\inf\left\{t>0:\ \int^t_0|Z(s)|^2ds\ge k\right\}.\]
		Applying It\^o's formula to $|Y(\cdot)|^2$, we have
		\[		\EE|Y(T\wedge\tau_k)|^2-|Y(0)|^2
		=-2\EE\int^{T\wedge\tau_k}_0\<Y,A^\top Y+C^\top Z+QX+\la\>ds+\EE\int^{T\wedge\tau_k}_0|Z|^2ds.\]
		Thus,
		\[
		\EE\int^{T\wedge\tau_k}_0|Z|^2ds\le\EE|Y(T\wedge\tau_k)|^2+K\EE\int^{T\wedge\tau_k}_0\(|Y|^2+|X|^2+|\la|^2\)ds
		+\frac12\EE\int^{T\wedge\tau_k}_0|Z|^2ds.\]
		Hence
		\[\EE\int^{T\wedge\tau_k}_0|Z|^2ds\le 2\EE|Y(T\wedge\tau_k)|^2+K\EE\int^{T\wedge\tau_k}_0\(|Y|^2+|X|^2+|\la|^2\)ds.\]
		Taking $k\to\infty$ on both sides, we obtain that
		\[\EE\int^{T}_0|Z|^2ds\le 2\EE|Y(T)|^2+K\EE\int^{T}_0\(|Y|^2+|X|^2+|\la|^2\)ds<\infty.\]
		Therefore, $\(X(\cdot),Y(\cdot),Z(\cdot) \) \in\(L^{2,c}_\FF(\RR^n)\)^2\times L^2_\FF(\RR^n)$ is a solution to FBSDE (\ref{linear1}). 		
		The uniqueness of $\(\phi(\cdot),\psi(\cdot)\)$ follows from that of $\(X(\cdot),Y(\cdot),Z(\cdot) \) $ in Lemma \ref{the3.4}.
				\end{proof}

The following theorem follows from Lemmas \ref{the3.4}, \ref{lem1115a} and \ref{lem0510a} directly.

\begin{theorem}\label{theorem371}
	Under Assumptions (H1) and (H2), the optimal control $u^*_{\al}(\cdot)$  of Problem 1 is given by 
	\begin{equation}
		\begin{aligned}\label{uu}
			u  ^*_{\al} =-\left(R+D^{\top} \Sigma D)^{-1}\left\{\left(B^{\top} \Sigma+D^{\top} \Sigma C+D^{\top} \Psi\right) \hat X+B^{\top} \hat \phi+D^{\top} \Sigma C_1 \alpha+D^{\top} \hat \psi\right\}\right.,
		\end{aligned}
	\end{equation}
	where $\hat X(\cdot)$ is the unique solution to (\ref{xu3}) (with $\big(\phi(\cdot), \psi(\cdot)\big) $ there replaced by $(\hat \phi(\cdot),\hat  \psi(\cdot))$), $\big(\Sigma(\cdot),  \Psi(\cdot)\big)$ is the unique solution of (\ref{R11}), and  $\big(\hat \phi(\cdot), \hat  \psi(\cdot)\big) $  is the unique solution of  (\ref{base211}) with $\lambda(\cdot)$ there  replaced by an arbitrary solution of  (\ref{xu5}).
\end{theorem} 

\begin{remark}
	Lemma \ref{lem0510a} and Theorem \ref{theorem371} have appeared in \cite{KT} for more general non-homogeneous terms. 
The conclusion there is even stronger than the current paper. More specifically, it was proved there that $\psi(\cdot) \in  L^{1,2}_{\FF}( \mathbb{R}^{n}) $. We include these results here for the reason that the problem here  involves  a special type  of nonhomogeneous terms, and our proof does not involve BMO-martingale theory which is quite technical.
\end{remark}
		 	 
		 	 \section{Optimal mean-field }\label{sec6}
 \setcounter{equation}{0}
\renewcommand{\theequation}{\thesection.\arabic{equation}}

In this section,  we present the proofs of Lemma \ref{lem1116b} and Theorem 
\ref{main0}. 

{\em Proof of Lemma \ref{lem1116b}}:
 Recalling the equations (\ref{xiong}, \ref{xiong2}, \ref{xiong3}),  the cost functional $\hat	J_{\alpha}\(u^*_\al (\cdot)\) $  can be written as 
\begin{equation}
	\begin{aligned}
	\hat	J_{\alpha}\(u^*_\al (\cdot)\) =&\;\mathbb{E} \Big\{ \int_{0}^{T}\Big(\langle QX,X\rangle+\langle Q_1\alpha,\alpha\rangle 
		+ \langle Ru^*_\al , u^*_\al \rangle   \Big)ds
		+	\langle GX(T),X(T)\rangle 	\Big\}\\
		=& \mathbb{E} \Big\{ \int_{0}^{T}\Big(\langle Q( \mathcal{P}_1x   +\mathcal{K}_1 \alpha), \mathcal{P}_1x   +\mathcal{K}_1 \alpha \rangle+\langle Q_1\alpha,\alpha\rangle \\
		&+ \langle R(\mathcal{P}_2x   +\mathcal{K}_2 \alpha) , \mathcal{P}_2x   +\mathcal{K}_2 \alpha\rangle   \Big)ds
		+	\langle G(\mathcal{P}_3x   +\mathcal{K}_3 \alpha), \mathcal{P}_3x   +\mathcal{K}_3\alpha \rangle 	\Big\} \\
		=&   \langle  (\mathcal{K}^*_1 Q  \mathcal{K}_1 + \EE Q_1 + \mathcal{K}^*_2R \mathcal{K}_2+ \mathcal{K}^*_3G \mathcal{K}_3) \al, \alpha \rangle_{\LL^2} \\
	&	+ 2 \langle ( \mathcal{K}^*_1 Q  \mathcal{P}_1   +\mathcal{K}^*_2 R \mathcal{P}_2+\mathcal{K}^*_3 G \mathcal{P}_3 ) x, \alpha\rangle_{\LL^2} \\
	&+  \langle  (\mathcal{P}^*_1 Q  \mathcal{P}_1 + \mathcal{P}^*_2R \mathcal{P}_2+ \mathcal{P}^*_3G \mathcal{P}_3) x, x \rangle_{\RR^n} .
	\end{aligned}			
\end{equation}
 Therefore,  applying Lemma \ref{lem1114a},  $\alpha^*(\cdot)$ is optimal for Problem 2 if and only if 
 $$
 (\mathcal{K}^*_1 Q  \mathcal{K}_1 +\EE  Q_1 + \mathcal{K}^*_2R \mathcal{K}_2+ \mathcal{K}^*_3G \mathcal{K}_3) \al ^* + ( \mathcal{K}^*_1 Q  \mathcal{P}_1   +\mathcal{K}^*_2 R \mathcal{P}_2+\mathcal{K}^*_3 G \mathcal{P}_3  ) x=  0.
 $$
We then obtain the conclusions of the Lemma  \ref{lem1116b}.
		\qed

		 	Finally, we present the
		 	
			{\em Proof of Theorem \ref{main0}}: It is easy to show that
		\[\inf_{u(\cdot) \in \mathcal{U}}J\big( u(\cdot)\big)= \inf_{\alpha(\cdot) \in \LL^2_0  }\;\; \inf_{u:\ \mathbb{E}X^u =\alpha} J\big( u(\cdot)\big).\]

		This justifies the decomposition of the original problem into Problems 1 and 2. The conclusion of Theorem  \ref{main0}
		then follows from those of Lemmas \ref{lemma3.3}, \ref{lem1115a},  \ref{lem1116b} and \ref{lem1114a}. 
		\qed

\section{Conclusion}\label{sec7}
\setcounter{equation}{0}
\renewcommand{\theequation}{\thesection.\arabic{equation}}

In this paper, we studied the MFSLQ control problem with random coefficients. We first established the existence and uniqueness of the optimal control $u^*(\cdot)$. Further, we characterized $u^*(\cdot)$ by an optimality system. However, this system contains terms such as $\EE[A_1(\cdot)^{\top}{Y}^*(\cdot)]$ 
which makes the explicit solution  very difficult to obtain. To overcome this hurdle, we decompose the original MFSLQ problem into Problems 
1 and 2.  Problem 1 is a usual SLQ control problem with the constraint that $\EE[X(t)]=\al(t)$ for all $t\in[0,T]$, where $\al(\cdot)$ is a deterministic function on $[0,T]$. We  introduced an ELM $\la(\cdot)$  to relax the constraint. Consequently, Problem 1 is reduced to a problem 
with control variables $u(\cdot)$ and $\la(\cdot)$ without any constraint for each fixed $\al(\cdot)$. After solving Problem 1,  Problem 2 becomes  a usual SLQ control problem with deterministic control variable $\al(\cdot)$.   
  We believe that the ELM method developed in this article can be used in other mean-field control or game problems. 
  For example, we can consider the
  conditional mean-field LQ control problem studied by \cite{Pham2016} and \cite{MWY} with random coefficients without assuming the coefficients to be 
  measurable with respect to the given filtration.

\vspace{1cm}

{\em Acknowledgement}: We would like to thank the AE and three anonymous referees for constructive suggestions which improve this article 
substantially. Especially, we thank the third referee who told us the reference \cite{KT} which we did not know when working on this project.
We also like to thank Zuoquan Xu from Hong Kong PolyU for fruitful discussions.

\end{document}